\documentclass[12pt]{amsart}
\usepackage{amsmath, amsthm, amsfonts, amssymb, float}
\usepackage[all]{xypic}
\usepackage[mathscr]{euscript}
\usepackage{tikz}
\usepackage{enumitem}

\usetikzlibrary{arrows}
\restylefloat{table}

\newtheorem{thm}{Theorem}[section]

\newtheorem{lem}[thm]{Lemma}

\newtheorem{defn}[thm]{Definition}
\newtheorem{conj}[thm]{Conjecture}

\newtheorem{ques}[thm]{Question}

\theoremstyle{remark}
\newtheorem*{remark}{Remark}

\newcommand{\ci}{\mathcal{I}}
\newcommand{\cz}{\mathcal{Z}}

\newcommand{\ca}{\mathcal{A}}

\newcommand{\cb}{\mathcal{B}}

\newcommand{\ct}{\mathscr{T}}

\newcommand{\C}{\mathbb{C}}

\newcommand{\srem}{s^R}

\newcommand{\llim}{\underleftarrow{lim}}
\newcommand{\ad}{\mathrm{Ad}_u}

\def\leqMvN{\preccurlyeq_{_\mathbf{M}}}
\def\lequ{\preceq_{_\mathbf{u}}}
\def\lessu{\prec_{_\mathbf{u}}}
\def\geqMvN{\succcurlyeq_{_\mathbf{M}}}

\def\V#1{V_{#1}}
\def\P#1{\Phi(#1)}
\def\Pf{\Phi}  
\def\Pzf{\Phi_z}  
\def\Pp#1{\Phi_p(#1)}  
\def\Ppf{\Phi_p}  
\def\II{\ci}
\def\TT{\ct}

\def\cat#1{\mathbf{\mathsf{#1}}}
\def\Spec{\mathrm{Spec}}
\def\Prim{\mathrm{Prim}}

\newcommand{\longhookrightarrow}{\ensuremath{\lhook\joinrel\relbar\joinrel\rightarrow}}
\def\setdef#1#2{\left\{#1\mid #2\right\}}             

\def\Mdot{ \text{ .}}
\def\Mcomma{ \text{ ,}}
\def\Msemicolon{ \text{ ;}}

\title{Partial and Total Ideals \\ of von Neumann Algebras}
\author{Nadish de Silva \quad\quad Rui Soares Barbosa}
\address{
Quantum Group, Department of Computer Science, University of Oxford \\
Wolfson Building\\
Parks Road\\
Oxford OX1 3QD \\
UK.}
\email{nadish.desilva@utoronto.ca}
\email{rui.soares.barbosa@cs.ox.ac.uk}

\usepackage{color}

\begin{document}

\pagestyle{empty}
\pagenumbering{gobble}

\binoppenalty=\maxdimen
\relpenalty=\maxdimen

\begin{abstract}
  A notion of \emph{partial ideal} for an operator algebra is a weakening the notion of ideal where the defining algebraic conditions are enforced only in the commutative subalgebras.
  We show that, in a von Neumann algebra, the ultraweakly closed two-sided ideals, which we call \emph{total ideals}, correspond to the unitarily invariant partial ideals.
  The result also admits an equivalent formulation in terms of central projections.
  We place this result in the context of an investigation into notions of spectrum of noncommutative $C^*$-algebras.\end{abstract}

\maketitle

\section{Introduction}

The principal theorem proved in this paper concerns the von Neumann algebraic analogue of a conjecture made for $C^*$-algebras.  We describe the original question concerning $C^*$-algebras, state the von Neumann algebraic version, and outline the more general investigation into noncommutative topology in the context of which the original question first arose.  In the next sections, we prove the principal theorem after establishing some technical preliminaries.  We conclude by indicating some ideas about how to tackle the $C^*$-algebraic case.

\subsection{Partial and total ideals of $C^*$-algebras}
All algebras and subalgebras considered throughout this paper are  assumed to be unital.  By a {\em total ideal} of a $C^*$-algebra $\ca$, we mean a norm closed, two-sided ideal of $\ca$.

\begin{defn} \label{pi}
A \emph{partial ideal} of a  $C^*$-algebra $\ca$ is a map $\pi$ that assigns to each commutative sub-$C^*$-algebra $V$ of $\ca$ a closed ideal of $V$ such that $\pi(V) = \pi(V') \cap V$ whenever $V \subset V'$.
\end{defn}

\begin{remark}
Suppose $\II$ is the contravariant functor from the category of $C^*$-algebras to the category of complete meet-semilattices which sends an algebra to its lattice of total ideals and a $*$-homomorphism $\phi: \ca \longrightarrow \cb$ to the homomorphism of complete meet-semilattices   $\II(\phi): \II(\cb) \longrightarrow \II(\ca)$ taking an ideal $I \subset B$ to the ideal $\phi^{-1}(I)$ of $\ca$.
Then a partial ideal is precisely a choice of $\pi(V) \in \II(V)$ for each commutative sub-$C^*$-algebra $V$ of $\ca$ such that whenever there is an inclusion morphism $\iota: V \longhookrightarrow V'$, then
\[\pi(V) = \II(\iota)(\pi({V'})) 
= \pi(V') \cap V \Msemicolon\]
i.e. the following diagram commutes.
\begin{equation}\label{diag:approx}
\xymatrix{
V' & & \{*\} \ar[rr]^{* \mapsto \pi(V')} \ar[ddrr]_{* \mapsto \pi(V)} & & \II(V') \ar[dd]^{\II(\iota)}
\\
&&&&
\\
V \ar@{^{(}->}[uu]^{\iota} & & & & \II(V)
}
\end{equation}
\end{remark}
\begin{remark}
  The concept of partial ideal was introduced by Reyes \cite{Reyes:obstructing} in the more general context of partial $C^*$-algebras.
  His definition differs slightly but is equivalent in our case:
  a subset $P$ of normal elements of $\ca$ such that $P \cap V$ is a closed ideal of $V$ for all commutative sub-$C^*$-algebras $V$ of $\ca$.
\end{remark}

Partial ideals exist in abundance: every closed, left (or right) ideal $I$ of $\ca$ gives rise to a partial ideal $\pi_I$ in a natural way by choosing $\pi_{I}(V)$ to be $I \cap V$.

For example, in a matrix algebra $M_n(\C)$, the right ideal $pM_n(\C)$, for $p \in M_n(\C)$ a nontrivial projection, yields a nontrivial partial ideal of $M_n(\C)$ in this way.
  As matrix algebras are simple, it cannot be the case that these nontrivial partial ideals also arise as $\pi_I$ from a total ideal $I$. 
This raises a natural question: 

\begin{ques} Which partial ideals of $C^*$-algebras arise from total ideals? \end{ques}

Some partial ideals do not even arise from left or right ideals: for example, choosing arbitrary nontrivial ideals from every nontrivial commutative sub-$C^*$-algebra of $M_2(\C)$ yields, in nearly all cases,\ nontrivial partial ideals of $M_2(\C)$.  However, a hint towards identifying those partial ideals which arise from total ideals is given by a simple observation.  If $\ad :\ca \longrightarrow \ca$ is an inner automorphism of $\ca$---that is, one given by conjugation by a unitary $u$ of $\ca$---then $\ad(I) = I$ for any total ideal $I \subset\ca$.  This imposes a special condition on the partial ideal $\pi_{I}(V) = I \cap V$ which arises from $I$.

\begin{defn} \label{ipi}
An {\em invariant partial ideal} $\pi$ of a $C^*$-algebra $\ca$ is a partial ideal of $\ca$ such that, for each commutative sub-$C^*$-algebra $V \subset \ca$ and any unitary $u \in \ca$, the rotation by $u$ of the ideal associated to $V$ is the ideal associated to the rotation by $u$ of $V$.  That is, 
\[\ad(\pi(V)) = \pi({\ad(V)}) \Mdot\]
\end{defn}

\begin{remark}
  Imposing the invariance condition on partial ideals is equivalent to extending the requirement on maps $\pi$ of Diagram \eqref{diag:approx} from inclusion maps
to all $*$-homomorphisms $\ad|_V^{V'}: V \longrightarrow V'$ arising as a restriction of the domain and codomain of an  inner automorphism.
An invariant partial ideal is precisely a choice of $\pi(V) \in \II(V)$ for each commutative sub-$C^*$-algebra $V$ of $\ca$ such that whenever there is a morphism $\ad|_V^{V'}: V \longrightarrow V'$ as above, then
  \[\pi(V) = \II(\ad|_{V}^{V'})(\pi({V'})) = \mathrm{Ad}_{u^*}(\pi(V')) \cap V \Msemicolon\]
  i.e. the following diagram commutes.
\[
\xymatrix{
V' & & \{*\} \ar[rr]^{* \mapsto \pi(V')} \ar[ddrr]_{* \mapsto \pi(V)} & & \II(V') \ar[dd]^{\II(\ad|_{V}^{V'})}
\\
&&&&
\\
V \ar[uu]_{\ad|_{V}^{V'}} & & & & \II(V)
}
\]

\end{remark}

%
%

Thus, we arrive at the simplest possible conjecture:

\begin{conj} \label{partialidealsc}
A partial ideal of a $C^*$-algebra arises from a total ideal if and only if it is an invariant partial ideal.
Consequently, the map $I \longmapsto \pi_I$ is a bijective correspondence between total ideals and invariant partial ideals.
\end{conj}
Note that the first part of the statement says that the map  $I \longmapsto \pi_I$ is surjective onto the invariant partial ideals.
The second part of the statement follows easily from this, since injectivity of this map is obvious: the left inverse is given by mapping an invariant partial ideal of the form $\pi_I$ to the linear span of $\bigcup_{V}\pi(V)$, which is equal to $I$ itself.

\subsection{Partial and total ideals of von Neumann algebras}

One may define \emph{partial ideal} (resp. \emph{invariant partial ideal}) for a  von Neumann algebra by replacing in Definition \ref{pi} (resp. Definition \ref{ipi})  the occurrences of ``commutative sub-$C^*$-algebra''with ``commutative sub-von-Neumann-algebra'' and ``closed ideal'' with ``ultraweakly closed ideal''.  A total ideal of a von Neumann algebra is an ultraweakly closed, two-sided ideal.
As before, a total ideal $I$ determines an invariant partial ideal $\pi_I$ in the same way, and the map $I \longmapsto \pi_I$ is injective.


Establishing the analogue of Conjecture \ref{partialidealsc} for von Neumann algebras
provides some measure of evidence for the original conjecture's verity,
and its proof may be adapted to show that the original conjecture holds for a large class of---or perhaps all---$C^*$-algebras.

\begin{thm}[Principal theorem] \label{main}
A partial ideal of a von Neumann algebra arises from a total ideal if and only if it is an invariant partial ideal.
Consequently, the map $I \longmapsto \pi_I$ is a bijective correspondence between total ideals and invariant partial ideals.
\end{thm}

Total ideals of a von Neumann algebra $\ca$ are in bijective correspondence with central projections $z$ of $\ca$: every total ideal $I$ is of the form $z\ca$ for a unique $z$ \cite{AlfsenShultz1}.  This allows us to rephrase the theorem in terms of projections which are vastly more convenient to work with.

\begin{defn}
A \emph{consistent family of projections} of a von Neumann algebra $\ca$ is a map $\Pf$ that assigns to each commutative sub-von-Neumann-algebra $V$ of $\ca$ a projection in $V$ such that:
\begin{enumerate}
  \item\label{item:dasein}
  for any $V$ and $V'$ such that $V \subset V'$, $\P{V}$ is the largest projection in $V$ which is less than or equal to $\P{V'}$, i.e.
  \[\P{V} = \sup\setdef{q \text{ is a projection in $V$}}{ q \leq \P{V'}} \Mdot\]
\end{enumerate}
An \emph{invariant family of projections} is such a map which further satisfies
\begin{enumerate}\setcounter{enumi}{1}
\item
  for any unitary element $u \in \ca$, $\P{uVu^*}$ = $u\P{V}u^*$.
\end{enumerate}
\end{defn}

The correspondence between total ideals and central projections yields correspondences between partial ideals (resp. invariant partial ideals) and consistent (resp. invariant) families of projections.  We therefore establish Theorem \ref{main} in the third section by proving an equivalent statement.
Just as was the case for ideals, any projection $p$ determines a consistent family of projections $\Ppf$ defined by choosing $\Pp{V}$ to be the largest projection $p$ in $V$ which is less than or equal to $p$.
For a central projection $z$, $\Pzf$ turns out to be an invariant family.
In the opposite direction, any consistent family of projections $\Pf$ gives a central projection  $\P{\cz(\ca})$ where $\cz(\ca)$ is the centre of $\ca$.

\begin{thm}[Principal theorem, reformulated] \label{main-reformulated}
A consistent family of projections of a von Neumann algebra arises from a central projection if and only if it is an invariant family of projections.
Consequently, the maps $z \longmapsto \Pzf$ and $\Pf \longmapsto \P{\cz(\ca)}$ define a bijective correspondence between central projections and invariant families of projections.
\end{thm}

\subsection{Motivation}

In this subsection, we motivate the characterization of total ideals as invariant partial ideals.

Previous work by the first author \cite{deSilva:DPhil-thesis}, inspired by ideas from the Isham-Butterfield approach to the Kochen-Specker theorem \cite{KochenSpecker} and the phenomenon of contextuality in quantum mechanics \cite{HamiltonIshamButterfield:topos3}, introduced a contravariantly functorial association of geometric objects $G(\ca)$ to  $C^*$-algebras $\ca$ which was meant to serve as a generalization of the notion of spectrum from the commutative to the noncommutative case.  It was shown that any functor $F: \cat{KHaus} \longrightarrow \cat{C}$ (from the category of compact Hausdorff spaces to a suitable target category $\cat{C}$) can be applied directly to the geometric objects given by $G$ to yield a functor $\tilde{F}: \cat{uC^*} \longrightarrow \cat{C}$ (from the category of unital $C^*$-algebras to $\cat{C}$).  The functor $\tilde{F}$ is an \emph{extension} of $F$ in the sense that, when restricted to the full subcategory of unital commutative $C^*$-algebras, $\tilde{F}|_{\cat{uComC^*}}$ is naturally isomorphic to $F$ composed with the Gel'fand spectrum functor.  The results of this method of extending topological concepts to $C^*$-algebraic ones can be compared with the constructions achieved via the canonical translation process of noncommutative geometry.  Agreement of the two methods could justify regarding $G(\ca)$ as the geometric manifestation of the noncommutative space underlying $\ca$.

It was demonstrated how the operator $K_0$ functor could be given a novel formulation in terms of $\tilde{K}$, the extension of topological $K$-theory \cite{deSilva:Ktheory}.  This raised the question of which other topological concepts can be automatically extended to their  noncommutative generalization by being applied directly to $G$.  It was conjectured that taking the extension of the notion of open subset  of a space (or equivalently, closed subset) would lead to the notion of closed, two-sided (i.e. total) ideal of a $C^*$-algebra.  To formalize this idea, let $\TT: \cat{KHaus} \longrightarrow \cat{CMSLat}$ be the functor which assigns to a compact Hausdorff topological space its complete lattice of closed sets (with $C_1 \leq C_2$ if and only if $C_1 \supset C_2$) and assigns to a continuous function the complete meet-semilattice homomorphism mapping a closed set to its image under the continuous function, let $\tilde{\TT}$ be its extension, and  let $\II$ be defined as in the remark following Definition \ref{pi}.

\begin{conj} \label{extenconj}
  The functors $\tilde{\TT}$ and $\II$ are naturally isomorphic.
\end{conj}

Formally, $G(\ca)$ is  a contravariant functor with, as its codomain, the category of compact Hausdorff spaces.  Consider the category $\cat{S}(\ca)$ whose objects are commutative sub-$C^*$-algebras of $\ca$ and whose morphisms are those which arise by restricting the domain and codomain of an inner automorphism of $\ca$ to such subalgebras.  Then $G(\ca)$ is the inclusion functor into $\cat{KHaus}$ of the image of $\cat{S}(\ca)$ under the Gel'fand spectrum functor.  The topological spaces in the diagram $G(\ca)$ should be thought of as being those which arise as quotient spaces of the hypothetical noncommutative space underlying $\ca$.  The morphisms of the diagram serve to track how these spaces fit together inside the noncommutative space underlying $\ca$.  The extension of a covariant functor $F: \cat{KHaus} \longrightarrow \cat{C}$ is defined as $\tilde{F} = \llim F G$.  Intuitively, one should think of the extension process as decomposing a noncommutative space into its quotient spaces, retaining those which are genuine topological spaces, applying the topological functor to each one, and pasting together the result.

Conjecture \ref{extenconj} essentially follows from Conjecture \ref{partialidealsc}.  To see this, note that the limit lattice $\tilde{\TT}(\ca) = \llim \TT G(\ca)$ is a cone over the diagram $\TT G(\ca)$.  
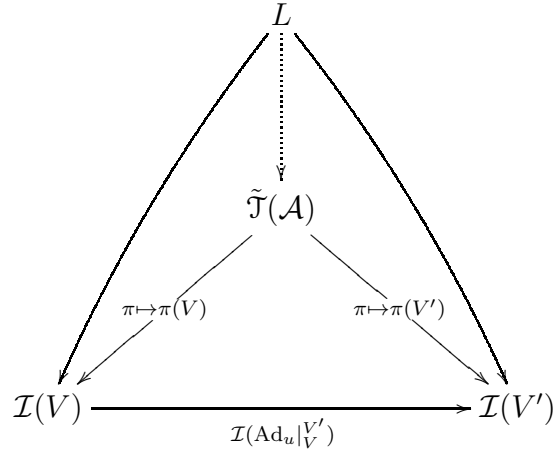
\[\xymatrix{
& & \ar@/_/[ddddll] \ar@/^/[ddddrr] L \ar@{.>}[dd] & &\\
&&&&\\
& & \ar[ddll]|{\pi \mapsto \pi(V)} \tilde{\TT}(\ca)  \ar[ddrr]|{\pi \mapsto \pi(V')} &&\\
&&&&\\
\II( V)  \ar[rrrr]_{\II(\ad|_{V}^{V'})}&&&& \II(V')
 }
\]

That is, the elements of $\tilde{\TT}(\ca)$ are precisely choices of elements from each $\II(V)$ subject to the condition of the remark following Definition \ref{ipi}. 

Proving Conjecture \ref{extenconj}  would establish a strong relationship between the topologies of the geometric object $G(\ca)$ and $\Prim(\ca)$, the primitive ideal space of $\ca$: we would be able to recover the lattice of the hull-kernel topology on $\Prim(\ca)$, as the limit of the topological lattices of the object $G(\ca)$.     (The primitive ideal space of $\ca$ is the set of ideals which are kernels of irreducible $*$-representations of $\ca$; in the commutative case, the primitive ideal space coincides with the Gel'fand spectrum.  It is a  $C^*$-algebraic variant of the ring-theoretic spectrum functor $\Spec$ which assigns  to a commutative ring a space whose hull-kernel topology provides the basis for sheaf-theoretic ring theory.)  Establishing this conjecture would allow considering $G$ to be an enrichment of $\Prim$.  This would open up the possibility of investigating the use of sheaf-theoretic methods in noncommutative geometry.

\section{Technical Preliminaries}

\subsection{Little lemmata}
In proving our main result, we shall make use of some simple properties of consistent families of projections which we record here as lemmata for clarity.

\begin{lem}\label{lemma:approxprops}
  Let $\ca$ be a von Neumann algebra, and $\Pf$ be a consistent family of projections.  Suppose $V \subset V'$ are commutative sub-von-Neumann-algebras of $\ca$. Then:
  \begin{enumerate}[label=(\arabic*)]
    \item\label{item:monotone} $\P{V} \leq \P{V'}$;
    \item\label{item:approxbest} if $p \in V$ and $p \leq \P{V'}$, then $p \leq \P{V}$;
    \item\label{item:approxbesteq} in particular, if $\P{V'} \in V$, then $\P{V'} = \P{V}$.
  \end{enumerate}
\end{lem}
\begin{proof}
  Properties \ref{item:monotone} and \ref{item:approxbest} are simply a rephrasing of the requirement in the definition of consistent family of projections that $\P{V}$ is the largest projection in $V$ smaller than $\P{V'}$. Property \ref{item:approxbesteq} is a particular case of \ref{item:approxbest}.
\end{proof}

Given a commutative subset $X$ of $\ca$,
denote by $\V{X}$ the commutative sub-von-Neumann-algebra of $\ca$ generated by $X$ and the centre $\cz(\ca)$,
i.e. $\V{X} = (X \cup \cz(\ca))''$.
Given a finite commutative set of projections $\{p_1,\ldots,p_n\}$, we write $\V{p_1, \ldots, p_n}$ for $\V{\{p_1, \ldots, p_n\}}$.

\begin{lem}\label{lemma:approxsupVm}
 Let $\ca$ be a von Neumann algebra; $\Pf$ a consistent family of projections; $M$ a commutative set of projections in $\ca$ such that $\P{\V{m}} \geq m$ for all $m \in M$; and  $s$ the supremum of the projections in $M$.
 Then $\P{\V{s}} \geq s$.
\end{lem}
\begin{proof}
    For all $m \in M$, since $\V{m} \subseteq \V{M}$, we have
  \[\P{\V{M}} \geq \P{\V{m}} \geq m\]
  by Lemma \ref{lemma:approxprops}-\ref{item:monotone} and the assumption that $\P{\V{m}} \geq m$.
  Hence, $\P{\V{M}}$ is at least the supremum of the projections in $M$, i.e. $\P{\V{M}} \geq s$.
  Now, from $\V{s} \subset \V{M}$ and $s \in \V{s}$, we conclude by Lemma \ref{lemma:approxprops}-\ref{item:approxbest} that $s \leq \P{\V{s}}$.
\end{proof}
 

\subsection{Partial orthogonality}

\begin{defn}
Two projections $p$ and $q$ are {\em partially orthogonal} whenever there exists a central projection $z$ such that $zp$ and $zq$ are orthogonal while $z^\perp p$ and $z^\perp q$ are equal.
\end{defn}

Note that partially orthogonal projections necessarily commute
and that if $p_1$ and $p_2$ are partially orthogonal, so is the pair $zp_1$ and $zp_2$ for any central projection $z$.
A set of projections is {\em partially orthogonal} whenever any pair of projections in the set is partially orthogonal.
We will require in the sequel the following simple lemma:

%

\begin{lem}\label{polemma}
Let $p_1$ and $p_2$ be projections and $z$ be a central projection such that
$zp_1$ and $zp_2$ are partially orthogonal and $z^\perp p_1$ and $z^\perp p_2$ are partially orthogonal. 
Then $p_1$ and $p_2$ are partially orthogonal.
\end{lem}
\begin{proof}
As $zp_1$ and $zp_2$ are partially orthogonal, there exists a central projections $y$ such that
\[ yzp_1 = yzp_2 \;\; \text{ and } \;\;         y^\perp zp_1 \perp y^\perp zp_2 \Mdot\]
Similarly, as $z^\perp p_1$ and $z^\perp p_2$ are partially orthogonal, there exists a central projections $x$ such that
\[ xz^\perp p_1 = xz^\perp p_2 \;\; \text{ and } \;\; x^\perp z^\perp p_1 \perp x^\perp z^\perp p_2 \Mdot\]
Summing both statements above, we conclude that
\[(yz+xz^\perp )p_1 = (yz+xz^\perp )p_2 \;\; \text{ and } \;\; (y^\perp z+x^\perp z^\perp ) p_1 \perp (y^\perp z + x^\perp z^\perp )p_2 \Mcomma\]
where $yz+xz^\perp $ is a central projection and $(yz+xz^\perp )^\perp  = y^\perp z+x^\perp z^\perp $.
So, $p_1$ and $p_2$ are partially orthogonal.
\end{proof}

\subsection{Main lemma}
When comparing projections, we write $\leq$ to denote the usual order on projections, $\leqMvN$ for the order up to Murray-von Neumann equivalence, and $\lequ$ for the order up to unitary equivalence.

The following lemma is one of the main steps of the proof.
The idea is to start with a projection $q$ in a von Neumann algebra and to
cover, as much as possible, its central carrier $C(q)$ by a commutative subset of the unitary orbit of $q$.
 The lemma states that, in order to cover $C(q)$ with projections from the unitary orbit of $q$,
it suffices to take a commutative subset, $M$, and (at most) one other projection, $uqu^*$, which is strictly larger than the remainder $C(q) - \sup M$.
That is, the remainder from what can be covered by a commutative set
 is strictly smaller than $q$ up to unitary equivalence.

\begin{lem}\label{lemma:main}
  Let $q$ be a projection in a von Neumann algebra $\ca$.
  Then there exists a set $M$ of projections such that:
\begin{enumerate}[label=(\arabic*)]
\item\label{mainlemmaitem:first}
$q \in M$;
\item
$M$ is a subset of the unitary orbit of $q$;
\item\label{mainlemmaitem:penultimate}
$M$ is a commutative set;
\item
the supremum $s$ of $M$ satisfies
  $$\srem \lessu q$$ 
where $\srem = C(q) - s$.
\end{enumerate}
%
\end{lem}
\begin{proof}
Let $O$ be the unitary orbit of $q$.  The partially orthogonal subsets of $O$  which contain $q$ form a poset under inclusion.
Given a chain in this poset, its union is partially orthogonal: any two projections in the union must appear together somewhere in one subset in the chain and are thus partially orthogonal. Hence, by Zorn's lemma, we can construct a maximal partially orthogonal subset $M$ of the unitary orbit of $q$ such that $q \in M$. Clearly, $M$ satisfies conditions \ref{mainlemmaitem:first}--\ref{mainlemmaitem:penultimate}. 

Denote by $s$ the supremum of the projections in $M$.
Its central carrier $C(s)$ is equal to the central carrier $C(q)$ of $q$.  This is because $C(-)$ is constant on unitary orbits and $C(\sup_{m \in M}m) = \sup_{m \in M}C(m)$.
We now need to show that $\srem \lessu q$.

By the comparison lemma for projections in a von Neumann algebra, there is a central projection $z$ such that
\[z\srem \geqMvN  zq \;\;\text{ and }\;\; z^\perp \srem \prec_{\mathbf{_M}} z^\perp q \Mdot\]
We can assume without loss of generality that $z \leq C(q)$ since $$C(q)^\perp q = C(q)^\perp \srem = 0 \Mdot$$
Moreover, as $s$ and $\srem$ are orthogonal, there is a unitary which witnesses these order relationships.
That is, there is a unitary $u$ such that 
\[z\srem\geq z(uqu^*)  \;\;\text{ and }\;\; z^\perp \srem < z^\perp (uqu^*) \Mdot\]
We will show that $z$ vanishes and thus conclude that $\srem < uqu^*$ as required.

Define $v$ to be the unitary $zu + z^\perp 1$ which acts as $u$ within the range of $z$ and as the identity on range of $z^\perp $.
We first establish that $vqv^*$ and $m$ are partially orthogonal for every $m \in M$.  

Let $m \in M$. As $M$ was defined to be a partially orthogonal set of projections and $q \in M$,
we know that $q$ and $m$ are partially orthogonal, and thus that $z^\perp q$ and $z^\perp m$ are partially orthogonal.
However, as $z^\perp v = z^\perp $, we may express this as: $z^\perp (vqv^*)$ and $z^\perp m$ are partially orthogonal.
Now, on the range of $z$, we have that
\[z(vqv^*) = z(uqu^*) \leq z\srem \;\;\text{ and }\;\; zm \leq zs \Mcomma\]
implying that $z(vqv^*)$ and $zm$ are orthogonal, hence partially orthogonal.
Putting both parts together, we have that $z^\perp vqv^*$ and $z^\perp m$ are partially orthogonal and that $z(vqv^*)$ and $zm$ are partially orthogonal.
We may thus apply Lemma \ref{polemma} and conclude that $vqv^*$ and $m$ are partially orthogonal as desired.


Having established that $vqv^*$ is partially orthogonal to all the projections in $M$, it follows by maximality of $M$ that $vqv^* \in M$.
Hence, \[z vqv^* \leq vqv^* \leq \sup M = s \Mdot\]
Yet, by construction, 
\[zvqv^* = zuqu^* \leq z\srem \leq \srem\Mcomma\] and so $zvqv^*$ must be orthogonal to $s$.
Being both contained within and orthogonal to $s$, $zvqv^*$ must vanish. Therefore, the unitarily equivalent projection $zq$ must also vanish.
Now, $z \leq C(q)$ and $zq = 0$ forces $z$ to be zero, for otherwise $C(q) - z$ would be both a central projection covering $q$ and also strictly smaller than the central carrier of $q$.
We may finally  conclude that $\srem < uqu^*$.
\end{proof}

\section{Main theorem}

Theorem \ref{main-reformulated}, and thus our principal result, Theorem \ref{main},
will follow as an immediate corollary of: 
\begin{thm}
  In a von Neumann algebra $\ca$, any invariant family of projections $\Pf$ arises from a central projection, i.e. $\Pf$ is equal to $\Pzf$ for the  central projection $z = \Pf(\cz(\ca))$.
\end{thm}
\begin{proof}
Let $\Pf$ be an invariant family of projections.
Suppose $W $ is a commutative sub-von-Neumann-algebra of $\ca$ which contains the centre $\cz(\ca)$, and let $q$ be the projection $\P{W}$.
We claim that $q$ is, in fact, equal to its own central carrier $C(q)$ and thus central.
As $q \leq C(q)$ is true by definition, we must show that $q \geq C(q)$.

We start by applying Lemma \ref{lemma:main} to  $q$. Let $M$ denote the resulting commuting set of projections in the unitary orbit of $q$,
$s$ denote the supremum of the projections in $M$, and $\srem$ denote $C(q) - s = C(s) - s$. From the lemma, we know that $\srem \prec_u q$, i.e. there exists a unitary $u$ such that $\srem < u q u^*$.


First note that, since $\V{q} \subset W$ and $q \in \V{q}$,  by Lemma \ref{lemma:approxprops}-\ref{item:approxbesteq}, we have that  $\P{\V{q}} = q$.
Then, by unitary invariance of the family of projections, for every $m \in M$ we have that $\P{\V{m}} = m$.
Hence, we can apply Lemma \ref{lemma:approxsupVm} to conclude that $\P{\V{s}}\geq s$.
 We also conclude, again by unitary invariance of $\Pf$, that $\P{\V{uqu^*}} = uqu^* \geq \srem$.

Now, note that $uqu^*$ and $\srem$ commute and that $\V{s} = \V{\srem}$. So there is a commutative sub-von-Neumann-algebra
$\V{s, uqu^*} \supseteq \V{s}, \V{uqu^*}$.
By Lemma \ref{lemma:approxprops}-\ref{item:monotone} and the two conclusions of the preceding  paragraph, we then have
\[\P{\V{s, uqu^*}} \geq \P{\V{s}} \vee \P{\V{uqu^*}} \geq s \vee \srem = C(q) \Mdot\]
But, since $C(q) \in \V{uqu^*}$ by virtue of being contained    in the centre, we can apply Lemma \ref{lemma:approxprops}-\ref{item:approxbest} to find that $\P{\V{uqu^*}} \geq C(q)$.
Finally, by unitary invariance, \[q = \P{\V{q}} \geq u^*C(q)u = C(q) \Mcomma\] concluding the proof that $q$ is central.


We have shown that the projection $\P{W}$ is central for every commutative sub-von-Neumann-algebra $W$ containing the centre $\cz(\ca)$.
By Lemma \ref{lemma:approxprops}-\ref{item:approxbesteq}, this means that $\P{W}$ is equal to $\P{\cz(\ca)}$, the projection chosen at the centre, for all such $W$. In turn, this determines the image of $\Pf$ on all commutative sub-von-Neumann-algebras $W'$ as 
\[\P{W'} = \sup\setdef{p \text{ is a projection in $W'$}}{p \leq \P{V_{W'\cup \cz(\ca)}} = \P{\cz(\ca)}} \Mcomma\]
and we find that $\Pf$ must be equal to $\Pf_{\P{\cz(\ca)}}$.
\end{proof}

\section{Conclusions}
Akemann and Pedersen \cite{pedersen} proposed to replace the translation process of noncommutative geometry by working directly with Giles and Kummer's \cite{GilesKummer} and Akemann's \cite{akemann} noncommutative generalizations of the basic topological notions of open and closed sets: open and closed projections in an enveloping von Neumann algebra.  In contrast, the framework proposed by the first author does not employ algebraic generalizations of basic topological notions, but rather, works with objects which generalize the notion of topological space and come readily equipped with an alternative to the translation process.
  Conjecture \ref{partialidealsc} is essentially the guess that the translation of the notion of closed set by this method matches up with the algebraic concept one would expect: closed, two-sided ideal.  It would also recover the hull-kernel topology on the primitive ideal space of a $C^*$-algebra$\ca$         as  a limit of the topologies; topologies of quotient spaces of the noncommutative space underlying $\ca$.

We have established the von Neumman algebraic analogue of Conjecture \ref{partialidealsc}.
 As a consequence, the original $C^*$-algebraic conjecture holds for all finite-dimensional $C^*$-algebras.  The question of whether it holds for all  $C^*$-algebras remains open. 
We conclude by indicating some ideas for future work that may lead to progress on this question.

One possible tack would be to enlarge the class of $C^*$-algebras for which the conjecture holds.  An immediate suggestion would be the class of $AF$-algebras which arise as limits of finite-dimensional $C^*$-algebras; it would follow immediately from a proof that $\tilde{T}$ preserves limits.

Another possibility would be to prove the whole conjecture directly by using the proof of the von Neumann algebraic version as a guide.  Indeed, one might still be able to reduce the question to one about projections by working in the enveloping von Neumann algebra $\ca^{**}$ of a $C^*$-algebra $\ca$. 
In this setting, the total ideals of a $C^*$-algebra $\ca$ correspond to certain total ideals of the enveloping algebra $\ca^{**}$ \cite{AlfsenShultz1}: those which correspond to open central projections. 
In essence, one would have to prove the appropriate analogue of Theorem \ref{main-reformulated} in order to find a correspondence between central open projections of $\ca^{**}$ and certain families of open projections which obey a restricted form of unitary invariance.

\section*{Acknowledgements}


It is our pleasure to thank Samson Abramsky, Bob Coecke, Andreas D\"oring, George Elliott, Chris Heunen, Kobi Kremnitzer, Manny Reyes, and participants of the Operator Algebras Seminar at the Fields Institute for their guidance, encouragement, and mathematical insights. 

NDS gratefully acknowledges  support from Merton College, the Oxford University Computing Laboratory, the Clarendon Fund, and the Natural Sciences and Engineering Research Council of Canada.

RSB gratefully acknowledges support from the
Marie Curie Initial Training Network MALOA -- From MAthematical LOgic to Applications, PITN-GA-2009-238381,
and from FCT -- Funda\c{c}\~{a}o para a Ci\^{e}ncia e Tecnologia (the Portuguese Foundation for Science and Technology),
PhD grant SFRH/BD/94945/2013.

\bibliographystyle{amsplain}
\bibliography{paperideals-refs}

 \end{document}